\documentclass[10pt,a4paper]{article}

\usepackage[utf8]{inputenc}
\usepackage[T1]{fontenc}
\usepackage{lmodern}
\usepackage[width=14.00cm,height=25.00cm]{geometry}
\usepackage{amsmath,amssymb,amsfonts,amsthm,mathtools}
\usepackage{enumerate}
\usepackage{microtype}
\usepackage{tikz}
\usepackage{hyperref}

\newtheorem{theorem}{Theorem}[section]
\newtheorem{proposition}[theorem]{Proposition}
\newtheorem{lemma}[theorem]{Lemma}
\newtheorem{corollary}[theorem]{Corollary}
\theoremstyle{definition}
\newtheorem{example}[theorem]{Example}
\theoremstyle{remark}
\newtheorem{remark}[theorem]{Remark}
\theoremstyle{plain}

\newcommand{\Fil}{\operatorname{Fil}}
\newcommand{\Cov}{\operatorname{Cov}}
\newcommand{\cl}{\operatorname{cl}}
\newcommand{\adh}{\operatorname{adh}}
\newcommand{\inh}{\operatorname{inh}}
\newcommand{\Cl}{\operatorname{Cl}}
\newcommand{\fin}{\mathrm{fin}}
\newcommand{\Alice}{\operatorname{Alice}}
\newcommand{\Bob}{\operatorname{Bob}}
\newcommand{\wins}{\mathrel{\uparrow}}
\newcommand{\notwins}{\mathrel{\not\uparrow}}

\title{Menger and Rothberger games on convergence spaces}
\author{Renan Mezabarba and Rodrigo Monteiro}
\date{}

\hypersetup{
  colorlinks=true,
  linkcolor=black,
  citecolor=black,
  urlcolor=black,
  pdftitle={Menger and Rothberger games on convergence spaces},
  pdfauthor={Renan Mezabarba and Rodrigo Monteiro},
  pdfsubject={Selection principles and games on convergence spaces},
  pdfkeywords={convergence space, selection principle, Menger game,
    Rothberger game}
}

\newcommand{\Addresses}{{
  \bigskip
  \footnotesize
  R.~Mezabarba, \textsc{Departamento de Ci\^encias Exatas, Universidade
  Estadual de Santa Cruz\\
  Ilh\'eus, BA 45662-900, Brazil}\par\nopagebreak
  \textit{E-mail address}, R.~Mezabarba:
  \texttt{rmmezabarba@uesc.br}

  \medskip
  R.~Monteiro, \textsc{Instituto de Ci\^encias Matem\'aticas e de
  Computa\c{c}\~ao, Universidade de S\~ao Paulo\\
  Avenida Trabalhador S\~ao-carlense 400, S\~ao Carlos,
  SP 13566-590, Brazil}\par\nopagebreak
  \textit{E-mail address}, R.~Monteiro:
  \texttt{rodrigosm@usp.br}
}}

\begin{document}

\maketitle

\begin{abstract}
We introduce Menger and Rothberger selection principles and games for
convergence spaces.  Alice plays families that meet every convergent filter,
and Bob's selections are required either to retain this property or merely to
cover the underlying set.  When the convergence is topological, both games
recover the classical games.  The winning condition
requiring an $L$-cover satisfies analogues of the Hurewicz and Pawlikowski
characterizations, the condition requiring a cover of $X$ is represented by
the weak Menger and Rothberger games.  We also show that, for a regular
convergence space, a winning strategy for Bob in
$\mathsf G_{\fin}(\mathcal C_L,\Cov(X))$ implies an Alster-type covering
property.  Under hereditary Lindel\"ofness the space is moreover a countable
union of compactoid subsets, which are compact in the pretopological case.
\end{abstract}



\section{Introduction}
\label{sec:introduction}

The modern theory of selection principles has its origins in Menger's 1924
covering property.  Hurewicz soon reformulated and studied that property, and
Rothberger isolated a stronger one in 1938
\cite{Menger1924,Hurewicz1926,Rothberger1938}.  In the notation later
systematized by Scheepers, these are
$\mathsf S_{\fin}(\mathcal O,\mathcal O)$ and
$\mathsf S_1(\mathcal O,\mathcal O)$, respectively
\cite{Scheepers1996}.  Their associated games are among the basic examples of
topological selection games.  The theorems of Hurewicz and Pawlikowski characterize
the two properties by
$\Alice\notwins\mathsf G_{\fin}(\mathcal O,\mathcal O)$ and
$\Alice\notwins\mathsf G_1(\mathcal O,\mathcal O)$, respectively
\cite{Hurewicz1926,Pawlikowski1994}.  The interaction between selections,
covers, and games remains central in the subject; see, for example,
\cite{SzewczakTsaban2020}.

The theory of convergence spaces originates in work of Choquet and Fischer
\cite{Choquet1948,Fischer1959}; we use the terminology of Dolecki and Mynard
\cite{DoleckiMynard2016}.  A convergence is specified directly by declaring
which filters converge to which points, rather than indirectly through a
topology.  Every topological space canonically determines a convergence space,
but a general convergence need not be determined by its open sets.  Covering
properties therefore require a formulation that explicitly retains
the convergent filters.

There is also a history of covering properties in convergence theory.  Feldman
used covering systems and basic subcoverings to formulate countability axioms
for convergence spaces \cite{Feldman1973}.  More recently, Mr\v{s}evi\'c and
Jeli\'c studied interior covers and their selection principles in \v{C}ech
closure spaces, together with related hyperspace constructions
\cite{MrsevicJelic2006,MrsevicJelic2008a,MrsevicJelic2008b}.  Since \v{C}ech
closure spaces give an equivalent description of pretopological convergence,
these papers provide a natural precedent for the present work.  Our aim is to
begin a systematic study of the corresponding games, starting with the
Menger and Rothberger cases for arbitrary convergences.

For a convergence space $(X,L)$, the relevant families are those that contain
a member of every filter with a nonempty $L$-limit set.  We call them
$L$-covers; they are also known as covering systems.  There are two natural
winning conditions: Bob's selections may be required to form an
$L$-cover, or merely to cover the underlying set.  Proposition~\ref{prop:topological-games}
shows that both choices recover the classical games whenever $L$ is induced by
a topology.  They differ for general convergences: Example~\ref{ex:branches}
gives a countable regular Hausdorff nonpretopological convergence space that
separates the two games in both their single- and finite-selection forms.

Our first principal representation, Theorem~\ref{thm:filter-games}, shows that,
for every convergence space $(X,L)$ and $\nu\in\{1,\fin\}$, the games
$\mathsf G_\nu(\mathcal C_L,\mathcal C_L)$ and
$\mathsf G_\nu(\mathcal C_L,\Cov(X))$ are strategically equivalent,
respectively, to the classical and weak covering games on a canonical
topological space of convergent filters.
Consequently, the Hurewicz and Pawlikowski characterizations hold without
additional hypotheses for $\mathsf G_\nu(\mathcal C_L,\mathcal C_L)$
(Corollary~\ref{cor:classical-transfer}).  For
$\mathsf G_\nu(\mathcal C_L,\Cov(X))$, Theorem~\ref{thm:weak-transfer} gives
the corresponding characterizations for the weak games under a Menger-type or
a countability hypothesis, respectively, and
Proposition~\ref{prop:strategywise-weak-menger} localizes the latter hypothesis
to a fixed strategy.

Finally, Theorem~\ref{thm:alster} shows that every regular convergence space
$(X,L)$ satisfying
\[
 \Bob\wins\mathsf G_{\fin}(\mathcal C_L,\Cov(X))
\]
has an Alster-type covering property~\cite{Alster1988}.  In the topological case this recovers
the theorem of Aurichi and Dias \cite{AurichiDias2014}.  If the convergence is
also hereditarily Lindel\"of in Feldman's sense, Theorem~\ref{thm:decomposition}
strengthens the conclusion to a countable union of $L$-compactoid subsets; for
regular pretopologies these subsets are compact
(Corollary~\ref{cor:compact-decomposition}).

\section{Background on selection games and convergence spaces}\label{sec:games}

Let $\mathcal A$ and $\mathcal B$ be collections of families of sets.  The
principle $\mathsf S_1(\mathcal A,\mathcal B)$ asserts that, for every sequence
$\langle\mathcal U_n:n<\omega\rangle$ of members of $\mathcal A$, there are
$U_n\in\mathcal U_n$ such that $\{U_n:n<\omega\}\in\mathcal B$.  In
$\mathsf S_{\fin}(\mathcal A,\mathcal B)$ one instead chooses finite
$\mathcal V_n\subseteq\mathcal U_n$ so that
$\bigcup_{n<\omega}\mathcal V_n\in\mathcal B$.

The corresponding games are denoted by
$\mathsf G_1(\mathcal A,\mathcal B)$ and
$\mathsf G_{\fin}(\mathcal A,\mathcal B)$.  In inning $n$, Alice chooses
$\mathcal U_n\in\mathcal A$.  Bob then chooses a member of $\mathcal U_n$ in
$\mathsf G_1(\mathcal A,\mathcal B)$, or a finite subfamily of $\mathcal U_n$
in $\mathsf G_{\fin}(\mathcal A,\mathcal B)$.  Bob wins if the total selection
belongs to $\mathcal B$.

When discussing a strategy for Bob, we record a finite position by the
corresponding finite sequence of Alice's moves.  If
$s=\langle\mathcal U_0,\ldots,\mathcal U_{n-1}\rangle$ and $\mathcal U$ is a
legal next move, then
$s\mathbin{\smallfrown}\langle\mathcal U\rangle$ is the sequence obtained by
adjoining $\mathcal U$ to $s$.  Thus, if $\sigma$ is a strategy for Bob,
$\sigma(s\mathbin{\smallfrown}\langle\mathcal U\rangle)$ denotes his response
to that position.  We write $\Alice\wins\mathsf G$ and $\Bob\wins\mathsf G$
when the indicated player has a winning strategy in $\mathsf G$; the notation
$\Alice\notwins\mathsf G$ and $\Bob\notwins\mathsf G$ has the corresponding
negative meaning.

We shall use the following general implication.  If
$\mathsf S_\nu(\mathcal A,\mathcal B)$ fails, a sequence witnessing its failure
prescribes a winning strategy for Alice: in inning $n$ she plays the $n$th
member of that sequence, independently of the earlier moves.  Hence
\[
 \Alice\notwins\mathsf G_\nu(\mathcal A,\mathcal B)
 \quad\Longrightarrow\quad
 \mathsf S_\nu(\mathcal A,\mathcal B),
 \qquad \nu\in\{1,\fin\}.
\]

Given a set $X$, a \emph{filter} on $X$ is a nonempty family
$\mathcal F\subseteq\mathcal P(X)$ that is closed under finite
intersections and upward closed, i.e., if $A\in\mathcal F$ and
$A\subseteq B\subseteq X$, then $B\in\mathcal F$.  A filter
$\mathcal F$ is \emph{proper} if $\varnothing\notin\mathcal F$.  An
\emph{ultrafilter} on $X$ is a proper filter that is maximal, with
respect to inclusion, among the proper filters on $X$. Equivalently,
$\mathcal U$ is an ultrafilter if and only if $A\in\mathcal U$ or
$X\setminus A\in\mathcal U$ for every $A\subseteq X$.
All filters in this paper are proper.  We write $\Fil(X)$ for the set
of filters on $X$, $\dot x$ for the principal ultrafilter at $x$
(i.e., $\dot x=\{A\subseteq X : x\in A\}$), and $A^\uparrow$ for the
principal filter generated by a nonempty set $A$ (i.e.,
$A^\uparrow=\{B\subseteq X : A\subseteq B\}$).  A filter $\mathcal G$
is finer than $\mathcal F$ when $\mathcal F\subseteq\mathcal G$.  Two
families $\mathcal A$ and $\mathcal B$ \emph{mesh}, written
$\mathcal A\mathrel{\#}\mathcal B$, if $A\cap B\ne\varnothing$
whenever $A\in\mathcal A$ and $B\in\mathcal B$.
When one argument is a set $A$, the notation $\mathcal F\mathrel{\#}A$
means that $F\cap A\ne\varnothing$ for every $F\in\mathcal F$.
A \emph{convergence} on $X$ is a function
\[
 L\colon\Fil(X)\longrightarrow\mathcal P(X)
\]
such that $x\in L(\dot x)$ for every $x\in X$ and
$L(\mathcal F)\subseteq L(\mathcal G)$ whenever $\mathcal G$ is finer
than $\mathcal F$.  We write $\mathcal F\to_Lx$ when
$x\in L(\mathcal F)$.  Every topology $\tau$ induces a convergence defined by
$
 \mathcal F\to_Lx$ if and only if $
 \mathcal N_\tau(x)\subseteq\mathcal F,
$
where $\mathcal N_\tau(x)$ is the neighborhood filter at $x$.
Conversely, the \emph{topological modification} $\tau_L$ consists of the sets
$O\subseteq X$ such that $O\in\mathcal F$ whenever $x\in O$ and
$\mathcal F\to_Lx$.

A convergence $L$ is a \emph{pretopology} if every $x\in X$ has a least
filter converging to $x$, denoted by $\mathcal V_L(x)$.  Equivalently,
$
 \mathcal F\to_Lx
 $ if and only if $
 \mathcal V_L(x)\subseteq\mathcal F.
$
Every topological convergence is a pretopology, with
$\mathcal V_L(x)=\mathcal N_\tau(x)$.

Consider the set of convergent filters
$
 \Sigma_L=\{\mathcal F\in\Fil(X):L(\mathcal F)\ne\varnothing\}
$
and denote by $\Cov(X)$ the family of all covers of $X$.  A family
$\mathcal U\subseteq\mathcal P(X)$ is an \emph{$L$-cover} if
$
 \mathcal U\cap\mathcal F\ne\varnothing
$
for every convergent filter $\mathcal F\in\Sigma_L$.
The family of all $L$-covers is denoted by $\mathcal C_L$.  Since
$\dot x\in\Sigma_L$ for every $x\in X$, one has
$\mathcal C_L\subseteq\Cov(X)$.

For a pretopology, put
$\inh_L(A)=\{x\in X:A\in\mathcal V_L(x)\}$.  In the equivalent language of
\v{C}ech closure spaces this is the interior of $A$, and $\mathcal U$ is an
$L$-cover exactly when $\{\inh_L(U):U\in\mathcal U\}$ covers $X$.  Thus these
are precisely the interior covers considered in
\cite{MrsevicJelic2006,MrsevicJelic2008a,MrsevicJelic2008b}.

We shall study
\[
 \mathsf G_\nu(\mathcal C_L,\mathcal C_L)
 \quad\text{and}\quad
 \mathsf G_\nu(\mathcal C_L,\Cov(X)),
 \qquad \nu\in\{1,\fin\}.
\]
The first asks that Bob's selections remain an $L$-cover, whereas the second
only asks that they cover $X$.  Since $\mathcal C_L\subseteq\Cov(X)$,
\[
 \Bob\wins\mathsf G_\nu(\mathcal C_L,\mathcal C_L)
 \Longrightarrow
 \Bob\wins\mathsf G_\nu(\mathcal C_L,\Cov(X)),
\]
and
\[
 \Alice\wins\mathsf G_\nu(\mathcal C_L,\Cov(X))
 \Longrightarrow
 \Alice\wins\mathsf G_\nu(\mathcal C_L,\mathcal C_L),
\]
for $\nu\in\{1,\fin\}$.

For a topology $\tau$ on $X$, let $\mathcal O(X,\tau)$ denote the family of
all open covers of $X$; we do not exclude covers that contain $X$ itself.

\begin{proposition}\label{prop:topological-games}
If $L$ is induced by a topology $\tau$ and $\nu\in\{1,\fin\}$, then each of
\[
 \mathsf G_\nu(\mathcal C_L,\mathcal C_L)
 \quad\text{and}\quad
 \mathsf G_\nu(\mathcal C_L,\Cov(X))
\]
is strategically equivalent to
$\mathsf G_\nu(\mathcal O(X,\tau),\mathcal O(X,\tau))$.  The same translations
identify the corresponding selection principles.
\end{proposition}

\begin{proof}
For every family $\mathcal U\subseteq\mathcal P(X)$,
\begin{equation}\label{eq:topological-l-cover}
 \mathcal U\in\mathcal C_L
 \quad\Longleftrightarrow\quad
 \{\operatorname{int}_\tau U:U\in\mathcal U\}
 \text{ covers }X.
\end{equation}
Indeed, if $\mathcal U$ is an $L$-cover, then it meets the convergent filter
$\mathcal N_\tau(x)$ for every $x\in X$; hence some $U\in\mathcal U$ is a
neighborhood of $x$.  Conversely, if $x\in\operatorname{int}_\tau U$ and
$\mathcal F\to_Lx$, then
$\operatorname{int}_\tau U\in\mathcal N_\tau(x)\subseteq\mathcal F$, and
therefore $U\in\mathcal F$.

Suppose first that Bob has a winning strategy in the classical game.  When
Alice plays an $L$-cover $\mathcal U$, Bob presents that strategy with the
open cover formed by the nonempty interiors of members of $\mathcal U$.  For
each distinct interior he retains one member of $\mathcal U$ having that
interior, and returns to the convergence game the members corresponding to
the open sets selected by the classical strategy.  Those open sets cover
$X$, so \eqref{eq:topological-l-cover} shows that the returned family is an
$L$-cover; in particular, it also covers $X$.  This gives Bob a winning
strategy for either convergence-space winning condition.

Conversely, suppose that Bob has a winning strategy in one of the
convergence-space games.  Alice's open covers may be played unchanged as
$L$-covers.  Bob's selected open sets cover $X$: this follows directly when
the target is $\Cov(X)$, and follows from
\eqref{eq:topological-l-cover} when the target is $\mathcal C_L$.  Thus Bob
has a winning strategy in the classical game.

Now suppose that Alice has a winning strategy in one of the convergence-space
games.  In the classical game she replaces each $L$-cover prescribed by her
strategy with its cover by nonempty interiors, retaining one original member
for each distinct interior.  She reports Bob's selected interiors to her
original strategy through these corresponding members.  If the latter fail
to form an $L$-cover, then the selected interiors fail to cover $X$ by
\eqref{eq:topological-l-cover}; if they fail merely to cover $X$, their
interiors fail to cover $X$ as well.  Hence the translated classical strategy
is winning.

Finally, if Alice has a winning strategy in the classical game, she may play
its open covers unchanged in either convergence-space game.  Any selected
family that fails to cover $X$ belongs neither to $\Cov(X)$ nor, by
\eqref{eq:topological-l-cover}, to $\mathcal C_L$.  These four translations
preserve finite histories and prove strategic equivalence.  Applied to a
sequence of covers rather than a strategy, the same translations prove the
assertion about selection principles.
\end{proof}

We now show how these two games can differ.

\begin{example}\label{ex:branches}
Let $X=2^{<\omega}$ be the set of finite binary sequences, let
$x_*=\varnothing$ be the empty sequence, and regard $X$ as the binary tree
rooted at $x_*$.  For $f\in2^\omega$ and $n<\omega$, put
\[
 C(f,n)=\{x_*\}\cup
         \{f\mathbin{\upharpoonright}m:m\ge n\},
\]
and let $\mathcal V_f$ be the filter generated by
$\{C(f,n):n<\omega\}$.  We define $L$ by declaring that the only filter
converging to a non-root node $s$ is $\dot s$, while
\[
 \mathcal F\to_Lx_*
 \quad\Longleftrightarrow\quad
 \mathcal V_f\subseteq\mathcal F
 \text{ for some }f\in2^\omega.
\]

For each $f\in2^\omega$, put
$B_f=\{f\mathbin{\upharpoonright}m:m<\omega\}$ and
$\mathcal U_*=\{B_f:f\in2^\omega\}$.  The family $\mathcal U_*$ is an
$L$-cover: indeed, $B_f=C(f,0)\in\mathcal V_f$ and every non-root node lies
on some branch.

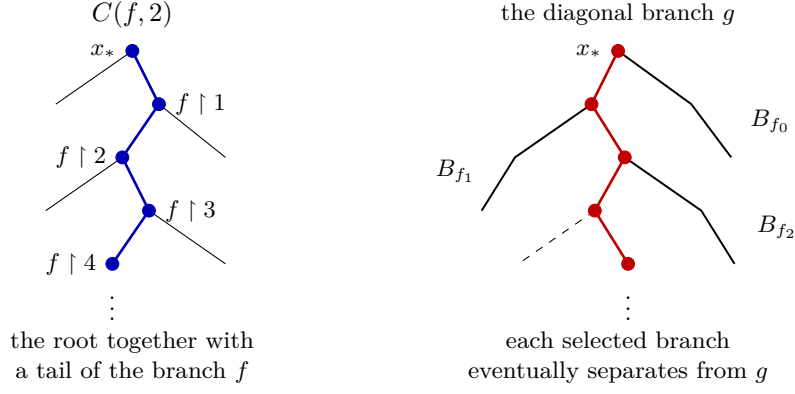
\begin{figure}[htbp]
\centering
\begin{tikzpicture}[
	scale=.88,
	every node/.style={font=\small},
	treeedge/.style={draw=black,line width=.4pt},
	chosen/.style={draw=blue!70!black,line width=1pt},
	escape/.style={draw=red!75!black,line width=1pt},
	dot/.style={circle,draw=black,fill=white,inner sep=1.5pt},
	cdot/.style={circle,draw=blue!70!black,fill=blue!70!black,inner sep=1.7pt},
	rdot/.style={circle,draw=red!75!black,fill=red!75!black,inner sep=1.7pt}
	]
	\node[font=\bfseries] at (-4.2,.55) {$C(f,2)$};
	\coordinate (l0) at (-4.2,0);
	\coordinate (l1) at (-3.8,-.8);
	\coordinate (l2) at (-4.35,-1.6);
	\coordinate (l3) at (-3.95,-2.4);
	\coordinate (l4) at (-4.5,-3.2);
	\draw[chosen] (l0)--(l1)--(l2)--(l3)--(l4);
	\draw[treeedge] (l0)--(-5.35,-.8);
	\draw[treeedge] (l1)--(-2.80,-1.6);
	\draw[treeedge] (l2)--(-5.5,-2.4);
	\draw[treeedge] (l3)--(-2.8,-3.2);
	\node[cdot,label=left:$x_*$] at (l0) {};
	\node[cdot,label=right:$f\upharpoonright1$] at (l1) {};
	\node[cdot,label=left:$f\upharpoonright2$] at (l2) {};
	\node[cdot,label=right:$f\upharpoonright3$] at (l3) {};
	\node[cdot,label=left:$f\upharpoonright4$] at (l4) {};
	\node at (-4.5,-3.75) {$\vdots$};
	\node[align=center,text width=5.0cm] at (-4.2,-4.6)
	{the root together with a tail of the branch $f$};

	\node at (3.1,.55) {the diagonal branch $g$};
	\coordinate (r0) at (3.1,0);
	\coordinate (r1) at (2.7,-.8);
	\coordinate (r2) at (3.2,-1.6);
	\coordinate (r3) at (2.75,-2.4);
	\coordinate (r4) at (3.25,-3.2);
	\draw[escape] (r0)--(r1)--(r2)--(r3)--(r4);
	\draw[treeedge,line width=.75pt] (r0)--(4.2,-.8)--(4.8,-1.6);
	\draw[treeedge,line width=.75pt] (r1)--(1.55,-1.6)--(1.05,-2.4);
	\draw[treeedge,line width=.75pt] (r2)--(4.35,-2.4)--(4.85,-3.2);
	\draw[treeedge,dashed] (r3)--(1.6,-3.2);
	\node[rdot,label=left:$x_*$] at (r0) {};
	\node[rdot] at (r1) {};
	\node[rdot] at (r2) {};
	\node[rdot] at (r3) {};
	\node[rdot] at (r4) {};
	\node at (3.25,-3.75) {$\vdots$};
	\node[anchor=west] at (4.95,-1.05) {$B_{f_0}$};
	\node[anchor=east] at (1.10,-1.80) {$B_{f_1}$};
	\node[anchor=west] at (5.05,-2.65) {$B_{f_2}$};
	\node[align=center,text width=5.2cm] at (3.1,-4.6)
	{each selected branch\\eventually separates from $g$};
\end{tikzpicture}
	\caption{The filters $\mathcal V_f$ contain the root and a tail of
		the branch $f$.  A countable family of selected branches omits some
		$g\in2^\omega$, and no selected $B_f$ belongs to $\mathcal V_g$.}
	\label{fig:branches}
\end{figure}

If Alice repeats $\mathcal U_*$, Bob selects only countably many branches,
even when finite selections are allowed.  Choose $g\in2^\omega$ different
from all of them.  Distinct branches eventually separate, so
$B_f\notin\mathcal V_g$ whenever $f\ne g$ (see
Figure~\ref{fig:branches}).  Thus the total selection is not an $L$-cover and
\[
 \Alice\wins\mathsf G_1(\mathcal C_L,\mathcal C_L)
 \quad\text{and}\quad
 \Alice\wins\mathsf G_{\fin}(\mathcal C_L,\mathcal C_L).
\]

On the other hand, enumerate $X$ as $\{s_n:n<\omega\}$.  In inning $n$,
Bob chooses from Alice's $L$-cover a set containing $s_n$.  These choices
cover $X$; hence
\[
 \Bob\wins\mathsf G_1(\mathcal C_L,\Cov(X))
 \quad\text{and}\quad
 \Bob\wins\mathsf G_{\fin}(\mathcal C_L,\Cov(X)).
\]
This convergence is Hausdorff, in the sense that every filter has at most one
limit point.  It is not pretopological: the filters $\mathcal V_f$ are pairwise
incomparable minimal filters converging to the root, and hence there is no
least filter converging to $x_*$.  Its regularity is verified in
Remark~\ref{rem:branch-regular}.
\end{example}

Example~\ref{ex:branches} suggests the natural first question: which classical
game theorems remain valid when open covers are replaced by $L$-covers?

\section{Topological representations}
\label{sec:representations}

We now represent both games on a canonical topological space associated with
$(X,L)$.  Under this representation,
$\mathsf G_\nu(\mathcal C_L,\mathcal C_L)$ becomes the classical covering
game, whereas $\mathsf G_\nu(\mathcal C_L,\Cov(X))$ becomes the corresponding
weak covering game.  This provides the main route by which classical game
theorems enter the theory.

For a subset $A\subseteq X$, put
\[
 A^\bullet=\{\mathcal F\in\Sigma_L:A\in\mathcal F\}
\]
Since $X^\bullet=\Sigma_L$ and
$A^\bullet\cap B^\bullet=(A\cap B)^\bullet$, these sets form a base for
a topology $\tau_L^\bullet$ on $\Sigma_L$.  Denote the resulting topological
space by
$
 Y_L=(\Sigma_L,\tau_L^\bullet),
$
and put $\Delta_X=\{\dot x:x\in X\}$.  The space $Y_L$ is appropriate for
representing $L$-covers as open covers.

\begin{proposition}\label{prop:filter-dictionary}
Let $(X,L)$ be a convergence space and $\mathcal U$ a family of subsets of
$X$.
\begin{enumerate}[(i)]
\item $\mathcal U$ is an $L$-cover if and only if
$\{A^\bullet:A\in\mathcal U\}$ covers $Y_L$.
\item The set $\Delta_X$ is dense in $Y_L$, and each of its points is
isolated.  Moreover,
\[
 \bigcup\mathcal U=X
 \quad\Longleftrightarrow\quad
 \bigcup_{A\in\mathcal U}A^\bullet
 \text{ is dense in }Y_L.
\]
\end{enumerate}
\end{proposition}

\begin{proof}
Assertion~(i) is the definition of an $L$-cover written in terms of
$A^\bullet$.  If $A^\bullet\ne\varnothing$, then $A\ne\varnothing$ and
$\dot x\in A^\bullet$ for each $x\in A$.  Hence every nonempty basic open
set meets $\Delta_X$.  Also
$\{x\}^\bullet=\{\dot x\}$, so these points are isolated.

An open subset of $Y_L$ is dense if and only if it contains $\Delta_X$:
necessity follows from the isolated singletons, and sufficiency from the
density of $\Delta_X$.  Finally,
$\dot x\in\bigcup_{A\in\mathcal U}A^\bullet$ exactly when
$x\in\bigcup\mathcal U$.  This proves~(ii).
\end{proof}

For a topological space $Y$, denote by $\mathcal O(Y)$ the family of its open
covers and by $\mathcal D(Y)$ the collection of families of open sets with
dense union.  Thus
$\mathsf G_\nu(\mathcal O(Y),\mathcal D(Y))$ is the weak Menger or weak
Rothberger game, according as $\nu=\fin$ or $\nu=1$.

In Theorem~\ref{thm:filter-games} we use the standard reduction to basic open
covers.  Given an open cover, take a basic refinement and retain, for each
basic set, one member of the original cover that contains it.  Replacing
selected basic sets by these containing members preserves both being a cover
and having dense union.  This transfers strategies, as well as selections, in
both directions.

\begin{theorem}\label{thm:filter-games}
Let $(X,L)$ be a convergence space and $\nu\in\{1,\fin\}$.
\begin{enumerate}[(i)]
\item The games
\[
 \mathsf G_\nu(\mathcal C_L,\mathcal C_L)
 \quad\text{and}\quad
 \mathsf G_\nu(\mathcal O(Y_L),\mathcal O(Y_L))
\]
are strategically equivalent.
\item The games
\[
 \mathsf G_\nu(\mathcal C_L,\Cov(X))
 \quad\text{and}\quad
 \mathsf G_\nu(\mathcal O(Y_L),\mathcal D(Y_L))
\]
are strategically equivalent.
\end{enumerate}
The same translations identify the corresponding selection principles.
\end{theorem}

\begin{proof}
Regard the basic open cover $\{A^\bullet:A\in\mathcal U\}$ as indexed by the
members $A$ of the $L$-cover $\mathcal U$, even when the map
$A\mapsto A^\bullet$ is not injective.  By
Proposition~\ref{prop:filter-dictionary}, the selected members of
$\mathcal U$ form an $L$-cover exactly when the associated basic open sets
cover $Y_L$, and they cover $X$ exactly when those basic open sets have dense
union.  This translation
preserves every finite history.  The reduction to basic covers completes the
proof.
\end{proof}

The classical characterizations now give the expected starting point for the
theory on convergence spaces.

\begin{corollary}\label{cor:classical-transfer}
For every convergence space $(X,L)$:
\begin{enumerate}[(i)]
\item $\mathsf S_{\fin}(\mathcal C_L,\mathcal C_L)$ holds if and only if
$\Alice\notwins
 \mathsf G_{\fin}(\mathcal C_L,\mathcal C_L)$;
\item $\mathsf S_1(\mathcal C_L,\mathcal C_L)$ holds if and only if
$\Alice\notwins
 \mathsf G_1(\mathcal C_L,\mathcal C_L)$.
\end{enumerate}
\end{corollary}

\begin{proof}
Apply Theorem~\ref{thm:filter-games}(i) and the theorems of Hurewicz and
Pawlikowski to $Y_L$.
\end{proof}

Part~(ii) of Theorem~\ref{thm:filter-games} identifies
$\mathsf G_\nu(\mathcal C_L,\Cov(X))$ with the weak topological games.  This
representation is unconditional.  Additional hypotheses enter only when one
asks whether the corresponding weak selection principle is characterized by
the nonexistence of a winning strategy for Alice.  By the general implication
noted in Section~\ref{sec:games}, only the converse direction requires proof.

The following transfer follows from the weak-game theorems of Babinkostova,
Pansera, and Scheepers~\cite{BabinkostovaPanseraScheepers2012}.  Their paper
adopts a standing $T_1$ convention, but the two proofs used below invoke no
separation axiom.

\begin{theorem}
\label{thm:weak-transfer}
Let $(X,L)$ be a convergence space.
\begin{enumerate}[(i)]
\item If $\mathsf S_{\fin}(\mathcal C_L,\mathcal C_L)$ holds, then
\[
 \mathsf S_1(\mathcal C_L,\Cov(X))
 \quad\Longleftrightarrow\quad
 \Alice\notwins\mathsf G_1(\mathcal C_L,\Cov(X)).
\]
\item If every $L$-cover contains a countable subfamily that is still an
$L$-cover, then
\[
 \mathsf S_{\fin}(\mathcal C_L,\Cov(X))
 \quad\Longleftrightarrow\quad
 \Alice\notwins\mathsf G_{\fin}(\mathcal C_L,\Cov(X)).
\]
\end{enumerate}
\end{theorem}

\begin{proof}
By Theorem~\ref{thm:filter-games}, the hypothesis in (i) says that $Y_L$ is
Menger.  Apply the weak Rothberger game theorem of Babinkostova, Pansera, and
Scheepers~\cite[Theorem~11]{BabinkostovaPanseraScheepers2012}.

For (ii), the stated hypothesis makes $Y_L$ Lindel\"of.  Indeed, a basic
open cover corresponds to an $L$-cover by
Proposition~\ref{prop:filter-dictionary}(i) and therefore has a countable
subcover.  An arbitrary open cover has a basic refinement.  The weak Menger
game theorem in
\cite[Theorem~28]{BabinkostovaPanseraScheepers2012} now applies.
\end{proof}

The countability assumption in part~(ii) can be localized to the particular
strategy under consideration.  This is useful when the representing space is
not Lindel\"of.

\begin{proposition}\label{prop:strategywise-weak-menger}
Assume $\mathsf S_{\fin}(\mathcal C_L,\Cov(X))$.  Let $\sigma$ be a
strategy for Alice in
$\mathsf G_{\fin}(\mathcal C_L,\Cov(X))$.  Suppose that, after every
finite position compatible with $\sigma$, the $L$-cover prescribed by
$\sigma$ contains a countable $L$-subcover.  Then $\sigma$ is not winning.
\end{proposition}

\begin{proof}
Transfer $\sigma$ to the game on basic open covers of $Y_L$.  Every cover
prescribed by the transferred strategy has a countable subcover.  Such a
subcover may be enumerated and replaced by the increasing cover formed by its
finite partial unions; a choice from the latter is returned as a finite choice
from the original cover.  The proof of
\cite[Theorem~28]{BabinkostovaPanseraScheepers2012} uses Lindel\"ofness only
for this normalization of the moves of the fixed strategy.  Its remaining
diagonalization uses
$\mathsf S_{\fin}(\mathcal O(Y_L),\mathcal D(Y_L))$, which is precisely
the selection principle assumed here.
\end{proof}

Neither global hypothesis in Theorem~\ref{thm:weak-transfer} is necessary.
The next elementary observation includes Example~\ref{ex:branches}.

\begin{proposition}\label{prop:countable-cov-game}
If $X$ is nonempty and countable, then
\[
 \Bob\wins\mathsf G_1(\mathcal C_L,\Cov(X))
 \quad\text{and}\quad
 \Bob\wins\mathsf G_{\fin}(\mathcal C_L,\Cov(X)).
\]
In particular, both corresponding selection principles hold.
\end{proposition}

\begin{proof}
Enumerate $X$ as $(x_n)_{n<\omega}$, with repetitions if necessary.  In
inning $n$, Bob chooses from Alice's $L$-cover a member containing $x_n$.
The resulting family covers $X$.  The same choices are legal in the
finite-selection game.
\end{proof}

\begin{remark}\label{rem:weak-boundary}
In Example~\ref{ex:branches}, the $L$-cover $\mathcal U_*$ has no countable
$L$-subcover.  Indeed, a countable subfamily omits a branch $g$, and none of
its members belongs to $\mathcal V_g$.  Thus $Y_L$ is not Lindel\"of and
hence is not Menger.  Nevertheless,
Proposition~\ref{prop:countable-cov-game} gives
$\Bob\wins\mathsf G_1(\mathcal C_L,\Cov(X))$.  Thus the hypotheses in both
parts of Theorem~\ref{thm:weak-transfer} are sufficient but not necessary.

The examples of Sakai~\cite{Sakai2014} do not contradict
Theorem~\ref{thm:weak-transfer}.  His regular Lindel\"of space that is not
weakly Menger only refutes the implication from Lindel\"ofness to weak
Mengerness.  Under CH, Sakai also constructs a Menger space on which Bob has no
winning strategy in $\mathsf G_{\fin}(\mathcal O,\mathcal D)$; this does not
show that Alice has a winning strategy.  Thus neither example supplies a
weakly Menger space with
$\Alice\wins\mathsf G_{\fin}(\mathcal O,\mathcal D)$.

Babinkostova, Pansera, and Scheepers ask whether such a topological space
exists.  Removing the hypothesis in Theorem~\ref{thm:weak-transfer}(ii) is a
more restricted issue: every representing space $Y_L$ has the dense set
$\Delta_X$ of isolated points.  A counterexample in the full topological
category would refute the convergence statement only if it can be realized,
up to the relevant games, within this class of representing spaces.
\end{remark}

\section{Compactoid and Alster-type consequences}
\label{sec:compactoid}

We now derive compactoid and Alster-type consequences from Bob's having a
winning strategy in $\mathsf G_{\fin}(\mathcal C_L,\Cov(X))$.  For
$A\subseteq X$, put
\[
 \cl_LA=\{x\in X:\mathcal F\to_Lx
        \text{ for some }\mathcal F\mathrel{\#}A\}.
\]
For a filter $\mathcal F\in\Fil(X)$, write
$\cl_L^\natural\mathcal F$ for the filter
generated by $\{\cl_LF:F\in\mathcal F\}$.  The convergence $L$ is
\emph{regular} if
\[
 \mathcal F\to_Lx
 \quad\Longrightarrow\quad
 \cl_L^\natural\mathcal F\to_Lx.
\]
This is the notion of regularity introduced by Cook and Fischer.  In the
topological case, it is the usual notion of regularity, without any separation
assumption~\cite{CookFischer1967}.

\begin{remark}\label{rem:branch-regular}
The convergence in Example~\ref{ex:branches} is regular.
For a non-root node $s$, one has $s\in\cl_LA$ exactly when $s\in A$.
At the root,
\[
 x_*\in\cl_LA
 \quad\Longleftrightarrow\quad
 \text{there is }f\in2^\omega\text{ such that }
 A\cap C(f,n)\ne\varnothing\text{ for every }n<\omega.
\]
It follows that every $C(f,n)$ is $L$-closed.  If
$\mathcal F\to_Lx_*$ and $\mathcal V_f\subseteq\mathcal F$, then
$\cl_L^\natural\mathcal F$ still contains $\mathcal V_f$, and hence
converges to $x_*$.  At a non-root node, the only convergent filter is
$\dot s$ and $\cl_L\{s\}=\{s\}$, so its closure filter is again $\dot s$.
\end{remark}

The \emph{adherence} of a filter $\mathcal H$ is
\[
 \adh_L\mathcal H
   =\bigcup\{L(\mathcal F):\mathcal F\mathrel{\#}\mathcal H\}
\]
A filter $\mathcal H$ is \emph{$L$-compactoid} if every ultrafilter finer
than $\mathcal H$ has a nonempty $L$-limit set.  A set $K\subseteq X$ is
$L$-compactoid if every ultrafilter $\mathfrak p$ with $K\in\mathfrak p$ has
a nonempty $L$-limit set, and it is \emph{$L$-compact} if every such
ultrafilter has an $L$-limit in $K$.  Finally, put
\[
 \Cl_L\mathcal H=\bigcap_{H\in\mathcal H}\cl_LH
\]

Proposition~\ref{prop:compactoid-hull} is the convergence-theoretic compactoid
hull needed in the sequel.  We include the passage through the pseudotopological
modification because the results in \cite{DoleckiGrecoLechicki1985} are
stated for pseudotopological spaces.

\begin{proposition}\label{prop:compactoid-hull}
Let $(X,L)$ be a regular convergence space and let $\mathcal H$ be an
$L$-compactoid filter.  Then
$\Cl_L\mathcal H$ is an $L$-compactoid subset of $X$ and
\[
 \adh_L\mathcal H\subseteq\Cl_L\mathcal H.
\]
If $(X,L)$ is also pretopological, then $\Cl_L\mathcal H$ is $L$-compact.
\end{proposition}

\begin{proof}
Let $SL$ be the pseudotopological modification of $L$, defined by
\[
 x\in SL(\mathcal F)
 \quad\Longleftrightarrow\quad
 x\in L(\mathfrak p)
 \text{ for every ultrafilter }\mathfrak p\supseteq\mathcal F
\]
The convergences $L$ and $SL$ coincide on ultrafilters.  Consequently, they
have the same compactoid filters and the same closure operator.  For the
latter assertion, one extends a filter meshing with a set to an ultrafilter
that contains that set.

Let $\mathfrak p$ be an $SL$-convergent ultrafilter.  It converges in
$L$, and regularity implies that $\cl_L^\natural\mathfrak p$ converges in
$L$; in particular this closure filter is $SL$-compactoid.  Thus every
convergent ultrafilter of $SL$ is subregular, meaning here that its closure
filter is $SL$-compactoid.  By
\cite[Theorems 4.2 and 4.4, Corollary 4.5]{DoleckiGrecoLechicki1985}, the
pretopological adherence of an $SL$-compactoid filter is compactoid.  Since
the closure operators coincide, that adherence is
$\Cl_{SL}\mathcal H=\Cl_L\mathcal H$.

If $x\in\adh_L\mathcal H$, some filter converges to $x$ and meshes with
$\mathcal H$; it therefore meshes with every $H\in\mathcal H$, so
$x\in\cl_LH$ for all $H$.  This proves the displayed inclusion.  Finally,
if $L$ is pretopological, $\Cl_L\mathcal H$ is the adherence of
$\mathcal H$, and its compactness follows from
\cite[Corollary 4.13]{DoleckiGrecoLechicki1985}.
\end{proof}

Fix a strategy $\sigma$ for Bob in
$\mathsf G_{\fin}(\mathcal C_L,\Cov(X))$.  If $s$ is a finite sequence of
moves by Alice and $\mathcal U\in\mathcal C_L$, set
\[
 B(s,\mathcal U)
   =\bigcup\sigma(s\mathbin{\smallfrown}\langle\mathcal U\rangle)
 \quad\text{and}\quad
 \mathcal B_s=\{B(s,\mathcal U):\mathcal U\in\mathcal C_L\}.
\]

\begin{lemma}\label{lem:strategy-filter}
If $\mathcal B_s$ has the finite intersection property, the filter
$\mathcal H_s$ generated by $\mathcal B_s$ is $L$-compactoid.
\end{lemma}

\begin{proof}
Let $\mathfrak p\supseteq\mathcal H_s$ be an ultrafilter.  For every
$\mathcal U\in\mathcal C_L$, the finite union $B(s,\mathcal U)$ belongs
to $\mathfrak p$, so some member of
$\sigma(s\mathbin{\smallfrown}\langle\mathcal U\rangle)$ belongs to
$\mathfrak p$.
Thus every $L$-cover has a member in $\mathfrak p$.

If $L(\mathfrak p)=\varnothing$, then no convergent filter is contained in
$\mathfrak p$.  For each $\mathcal F\in\Sigma_L$, choose
$A_{\mathcal F}\in\mathcal F\setminus\mathfrak p$.  The family
$\{A_{\mathcal F}:\mathcal F\in\Sigma_L\}$ is an $L$-cover with no member
in $\mathfrak p$, a contradiction.
\end{proof}

\begin{theorem}\label{thm:alster}
Let $(X,L)$ be a regular convergence space such that
\[
 \Bob\wins\mathsf G_{\fin}(\mathcal C_L,\Cov(X))
\]
If $\mathcal W$ is a cover of $X$ by $G_\delta$-subsets of $(X,\tau_L)$ and every
$L$-compactoid subset of $X$ is contained in some member of $\mathcal W$,
then $\mathcal W$ has a countable subcover.
\end{theorem}

\begin{proof}
Assume $X\ne\varnothing$, fix a winning strategy $\sigma$ for Bob, and
let $\mathcal W$ be as in the statement.
For a finite sequence $s$ of moves by Alice, first suppose that
$\mathcal B_s$ has the finite intersection property.  Let $\mathcal H_s$
be the generated filter and put $K_s=\Cl_L\mathcal H_s$.  By
Lemma~\ref{lem:strategy-filter} and
Proposition~\ref{prop:compactoid-hull}, $K_s$ is $L$-compactoid.  Choose
$W_s\in\mathcal W$ with $K_s\subseteq W_s$ and write
\[
 W_s=\bigcap_{n<\omega}O_{s,n},
\]
where each $O_{s,n}$ is $\tau_L$-open.

We claim that $O_{s,n}\in\mathcal H_s$ for every $n$.  Otherwise, extend
$\mathcal H_s\cup\{X\setminus O_{s,n}\}$ to an ultrafilter
$\mathfrak p$.  Compactoidness gives $x\in L(\mathfrak p)$.  Since
$\mathfrak p\mathrel{\#}\mathcal H_s$,
$x\in\adh_L\mathcal H_s\subseteq K_s\subseteq O_{s,n}$.  The definition
of $\tau_L$ then gives $O_{s,n}\in\mathfrak p$, a contradiction.

For each $n$, choose $L$-covers
$\mathcal U_{s,n,1},\ldots,\mathcal U_{s,n,k(s,n)}$ such that
\begin{equation}\label{eq:alster-node}
 \bigcap_{i=1}^{k(s,n)}B(s,\mathcal U_{s,n,i})
 \subseteq O_{s,n}
\end{equation}
Declare all these covers to be possible next moves after $s$.

If $\mathcal B_s$ does not have the finite intersection property, choose
$L$-covers $\mathcal U_{s,1},\ldots,\mathcal U_{s,m(s)}$ with
\begin{equation}\label{eq:empty-node}
 \bigcap_{i=1}^{m(s)}B(s,\mathcal U_{s,i})=\varnothing,
\end{equation}
declare them to be the possible next moves, and choose any
$W_s\in\mathcal W$.

Starting with the empty sequence and adjoining the possible next moves at
each node produces a countably branching tree $T$ of finite sequences of
Alice's moves.
Every level, and hence $T$, is countable.  We show that
$\{W_s:s\in T\}$ covers $X$.

Suppose that $x$ belongs to none of these sets.  At a node of the first
type, choose $n$ with $x\notin O_{s,n}$.  Equation
\eqref{eq:alster-node} gives an $i$ such that
$x\notin B(s,\mathcal U_{s,n,i})$.  At a node of the second type,
Equation~\eqref{eq:empty-node} gives the same conclusion for some $i$.
Choosing the corresponding successor at every stage produces a play in
which Bob follows $\sigma$ and every response misses $x$.  This contradicts
the fact that $\sigma$ is winning.
\end{proof}

For a regular topological space, compactoid subsets are exactly the subsets
with compact closure~\cite[Proposition 3.12 and Corollary
4.7]{DoleckiGrecoLechicki1985}.
Therefore Theorem~\ref{thm:alster} says precisely that $X$ is an Alster
space and recovers the result of
Aurichi and Dias~\cite[Corollary 2.13]{AurichiDias2014}.

\begin{corollary}\label{cor:pretopological-alster}
Let $(X,L)$ be a regular pretopological space.  If
$\Bob\wins\mathsf G_{\fin}(\mathcal C_L,\Cov(X))$, then every cover of $X$ by
$G_\delta$-subsets of $(X,\tau_L)$ such that every $L$-compact subset is
contained in some member of the cover has a countable subcover.
\end{corollary}

\begin{proof}
The empty set is $L$-compact.  If $A\subseteq X$ is a nonempty
$L$-compactoid subset, then $A^\uparrow$ is an $L$-compactoid filter, so
Proposition~\ref{prop:compactoid-hull} shows that $\Cl_L(A^\uparrow)$ is
$L$-compact and contains $A$.  Thus the cover in
the statement also satisfies the hypothesis of Theorem~\ref{thm:alster}.
\end{proof}

We next use the countability condition introduced by
Feldman~\cite{Feldman1973}.  A subfamily $\mathcal V$ of an $L$-cover
$\mathcal U$ is a \emph{basic subcovering} if, for every
$\mathcal F\in\Sigma_L$, some finite $\mathcal E\subseteq\mathcal V$ satisfies
$\bigcup\mathcal E\in\mathcal F$.  The convergence space is
\emph{Lindel\"of} if every $L$-cover has a countable basic subcovering.  For a
topological convergence this is the usual Lindel\"of property: apply the
definition to open covers in one direction, and take interiors of members of
an $L$-cover in the other.

\begin{remark}
The convergence space of Example~\ref{ex:branches} is not Lindel\"of.  Indeed,
given countably many members $B_{f_k}$ of $\mathcal U_*$, choose
$g\notin\{f_k:k<\omega\}$.  No finite union of the selected branches contains
a tail $C(g,n)$, so no such union belongs to $\mathcal V_g$.  On the other
hand, Remark~\ref{rem:branch-regular} and Theorem~\ref{thm:alster} show that
this convergence satisfies the Alster-type conclusion of that theorem.  Thus,
in the setting of convergence spaces, this conclusion does not imply
Lindel\"ofness.
\end{remark}

For $Y\subseteq X$ and a filter $\mathcal F\in\Fil(Y)$, define 
\[
 \mathcal F^{\uparrow X}
 =\{A\subseteq X:A\cap Y\in\mathcal F\}
\]
The subspace convergence $L|_Y$ is defined by
$\mathcal F\to_{L|_Y}y$ if and only if
$\mathcal F^{\uparrow X}\to_Ly$.  We say that $(X,L)$ is
\emph{hereditarily Lindel\"of} when every subspace is Lindel\"of in
Feldman's sense.

\begin{theorem}\label{thm:decomposition}
Let $(X,L)$ be a regular, hereditarily Lindel\"of convergence space.  If
$\Bob\wins\mathsf G_{\fin}(\mathcal C_L,\Cov(X))$, then $X$ is a countable
union of $L$-compactoid subsets.
\end{theorem}

\begin{proof}
Assume $X\ne\varnothing$ and fix a winning strategy $\sigma$ for Bob.
For a finite sequence $s$ of moves by Alice, suppose first that
$\mathcal B_s$ has the finite intersection property.  Let $\mathcal H_s$ be
the generated filter, put $K_s=\Cl_L\mathcal H_s$, and let
$Y_s=X\setminus K_s$.  For $H\in\mathcal H_s$, put $P_H=Y_s\setminus H$.
The set $K_s$ is
$L$-compactoid by Lemma~\ref{lem:strategy-filter} and
Proposition~\ref{prop:compactoid-hull}.

The family $\{P_H:H\in\mathcal H_s\}$ is an $L|_{Y_s}$-cover.  Indeed,
let $\mathcal F\to_{L|_{Y_s}}y$.  Since $y\notin K_s$, choose
$H\in\mathcal H_s$ with $y\notin\cl_LH$.  The filter
$\mathcal F^{\uparrow X}$ cannot mesh with $H$, so
$P_H\in\mathcal F$.

By hereditary Lindel\"ofness, choose a countable basic subcovering
$\{P_{H_{s,n}}:n<\omega\}$.  For every $n$, choose $L$-covers
$\mathcal U_{s,n,1},\ldots,\mathcal U_{s,n,k(s,n)}$ such that
\begin{equation}\label{eq:lindelof-node}
 \bigcap_{i=1}^{k(s,n)}B(s,\mathcal U_{s,n,i})
 \subseteq H_{s,n},
\end{equation}
and declare them to be possible next moves after $s$.

If $\mathcal B_s$ does not have the finite intersection property, put
$K_s=\varnothing$ and choose finitely many $L$-covers
$\mathcal U_{s,1},\ldots,\mathcal U_{s,m(s)}$ such that
\[
 \bigcap_{i=1}^{m(s)}B(s,\mathcal U_{s,i})=\varnothing.
\]
Declare these covers to be the possible next moves.

As in the proof of Theorem~\ref{thm:alster}, these choices generate a
countable tree $T$ of finite sequences of Alice's moves.  Suppose that
$x\notin\bigcup_{s\in T}K_s$.  At a node of the first type, the principal
ultrafilter at $x$ on $Y_s$ converges to $x$.  Since
$\{P_{H_{s,n}}:n<\omega\}$ is a basic subcovering, one has
$x\in P_{H_{s,n}}$ for some $n$, and hence $x\notin H_{s,n}$.
Equation~\eqref{eq:lindelof-node} supplies an $i$ such that
$x\notin B(s,\mathcal U_{s,n,i})$.  At a node of the second type, the
empty intersection supplies such an $i$ as well.

Following these successors yields a play in which Bob follows $\sigma$ but
never covers $x$, a contradiction.  Thus
$X=\bigcup_{s\in T}K_s$, a countable union of $L$-compactoid subsets.
\end{proof}

\begin{corollary}\label{cor:compact-decomposition}
Let $(X,L)$ be a regular, hereditarily Lindel\"of pretopological space.  If
\[
 \Bob\wins\mathsf G_{\fin}(\mathcal C_L,\Cov(X)),
\]
then $X$ is a countable union of $L$-compact subsets.
\end{corollary}

\begin{proof}
At every node of the proof of Theorem~\ref{thm:decomposition},
$K_s$ is $L$-compact by Proposition~\ref{prop:compactoid-hull} when
$\mathcal B_s$ has the finite intersection property; at the remaining nodes
$K_s=\varnothing$.
\end{proof}

For regular hereditarily Lindel\"of topological spaces,
Corollary~\ref{cor:compact-decomposition} is the familiar implication
\[
 \Bob\wins\mathsf G_{\fin}(\mathcal C_L,\Cov(X))
 \quad\Longrightarrow\quad X\text{ is }\sigma\text{-compact}.
\]
Conversely, if
$X=\bigcup_{n<\omega}K_n$ with each $K_n$ compact, then, in inning $n$, Bob
chooses finitely many members of Alice's $L$-cover whose interiors cover
$K_n$.  Thus $\Bob\wins\mathsf G_{\fin}(\mathcal C_L,\Cov(X))$.

\section*{Acknowledgements}

The second author was financed in part by the Coordena\c{c}\~ao de Aperfei\c{c}oamento
de Pessoal de N\'ivel Superior -- Brasil (CAPES) -- Finance Code 001.

\bigskip
\noindent\textbf{Declaration of generative AI and AI-assisted technologies
in the writing process.}\par
\smallskip

\noindent During the preparation of this work, the authors used OpenAI’s Codex for the discussion
of ideas and for assistance with the writing of the manuscript. The authors reviewed and
edited the resulting text and take full responsibility for the content of the article.

\Addresses

\end{document}